\documentclass[11pt,
twoside]{amsart}
\usepackage{wrapfig}
\usepackage{graphicx}
\usepackage{color}
\usepackage{amsmath}
\usepackage{amssymb}
\usepackage{amsfonts}
\usepackage{latexsym, amsthm, mathrsfs}
\usepackage{hyperref, breakurl}
\usepackage[latin1]{inputenc}
\usepackage[T1]{fontenc} 
\usepackage{lmodern} 
\usepackage{fancybox}
\usepackage{amscd}
\input{xy}
\usepackage[all]{xy}

\numberwithin{equation}{section}
\newtheorem{theorem}{Theorem}[section]
\newtheorem{lemma}[theorem]{Lemma}

\theoremstyle{definition} 
\newtheorem{definition}[theorem]{Definition}

\theoremstyle{plain} 

\newtheorem{proposition}[theorem]{Proposition}

\newtheorem*{maintheorem*}{Main Theorem}
\newtheorem*{conjecture*}{Conjecture}
\newtheorem*{theorem*}{Theorem}
\newtheorem*{proposition*}{Proposition}
\newtheorem*{corollary*}{Corollary}
\newtheorem{observation}[theorem]{Observation}
\theoremstyle{remark}  
\newtheorem{remark}[theorem]{Remark}
\newtheorem*{remarks*}{Remarks}
\newtheorem*{remark*}{Remark}
\newtheorem*{claim*}{Claim}

\newtheoremstyle{mystyle}
  {3pt}{3pt}{\itshape}{}{\bfseries}{}{.5em}
  {\thmname{#1}\thmnumber{ #2}\thmnote{. #3}}
\theoremstyle{mystyle}
\newtheorem*{maintheoremnew*}{Main Theorem}
\newtheorem*{conjecturenew*}{Conjecture}
\newtheorem*{theoremnew*}{Theorem}
\newtheorem*{propositionnew*}{Proposition}

\newcommand{\nc}{\newcommand}
\nc{\nothing}[1]{}
\nothing{   
\nc{\dom}{{\rm dom}}
\nc{\card}{{\rm card}}
\nc{\lh}{{\rm lh}}
\nc{\lgg}{{\rm lg}}
\nc{\rge}{\mbox{\rm range}}
\nc{\cf}{{\rm cf}}
nc{\nex}{\mbox{\rm next}}
\nc{\uhr}{\restriction}
\nc{\supt}{{\rm supt}}
\nc{\supp}{{\rm supp}}
\nc{\Lim}{{\rm Lim}}
\nc{\Leb}{{\rm Leb}}
\nc{\modd}{{\rm mod}}
\nc{\RO}{{\rm RO}}
\nc{\prob}{{\rm Prob}}
}

\nc{\On}{{\rm On}}
\nc{\Ord}{{\rm On}}
\nc{\nco}{\DeclareMathOperator}
\nco{\rk}{rk}
\nco{\order}{o}
\nco{\ppower}{pp}
\nco{\pcf}{pcf} 
\nco{\tcf}{tcf} 
\nco{\tlim}{tlim} 
\nco{\limtext}{lim} 
\nco{\prodt}{{\textstyle \prod}}
\nco{\symdiff}{\triangle}
\nco{\dom}{dom}
\nco{\card}{card}
\nco{\lh}{lh}
\nco{\lt}{lt}
\nco{\lgg}{lg}
\nco{\hgt}{ht}
\nco{\rge}{range}
\nco{\otp}{otp}
\nco{\trunk}{tr}
\nco{\cf}{cf}
\nco{\nex}{next}
\nc{\uhr}{\restriction}
\nco{\reduction}{red}
\nco{\supt}{supt}
\nco{\supp}{supp}
\nco{\Lim}{Lim}
\nco{\Leb}{Leb}
\nco{\modd}{mod}
\nco{\invariant}{inv}
\nco{\id}{id}
\nco{\RO}{RO}
\nco{\poss}{pos}
\nco{\Inc}{Inc} 
\nco{\Ge}{Ge}
\nco{\hdrop}{\hat{drop}}
\nco{\Gen}{Gen}
\nco{\Property}{Pr}
\nco{\GT}{GT}
\nco{\refl}{refl}
\nco{\nmg}{nmg}
\nco{\en}{en}
\nco{\truth}{truth} 
\nco{\sw}{sw} 

\nc{\potom}{\ensuremath{{\cal P}(\omega)}}
\nc{\potinf}{\ensuremath{[\omega]^\omega}}
\nc{\pfin}{\ensuremath{{\cal P}(\omega)/{\rm fin}}}

\nc{\potfin}{\ensuremath{[\omega]^{<\omega}}}
\nc{\inn}{\ensuremath{{\omega^{\uparrow \omega}}}}
\nc{\baire}{{}^\omega \omega}
\nc{\bair}{\omega^\omega}
\nc{\hoch}{^{<\omega}}
\nc{\hocho}{^{\omega}}
\nc{\tree}[1]{{[} #1 {]}_0}
\nc{\tre}[2]{ {#1}_{#2}}

\nc{\prooff}[1]{{\bf Proof} of #1:}
\nc{\proofend}{\makebox{} \hfill ${\bf \square}$ \\}
\nc{\proofendof}[1]{\makebox{} \hfill ${\boldmath{\square}}_{\rm #1}$ \\}
\nc{\beq}{\begin{eqnarray*}}
\nc{\eeq}{\end{eqnarray*}}
\nc{\bde}{\begin{list}}
\nc{\ede}{\end{list}}

\newenvironment{myrules}
{\begin{list}{}
{
 \setlength{\leftmargin}{0.8cm}
 \setlength{\labelwidth}{0.8cm}
 \setlength{\labelsep}{0.2cm}
 \setlength{\parsep}{0.5ex plus 0.2ex minus 0.1 ex}
 \setlength{\itemsep}{0.3ex plus 0.2 ex minus 0ex}
}}{\end{list}}

{\end{list}}

\newcounter{subalph}
{\end{list}}

\newcommand{\greek}[1]{\ifthenelse{\value{#1}=1}{\mbox{$\alpha$}}%
  {\ifthenelse{\value{#1}=2}{\mbox{$\beta$}}{%
   \ifthenelse{\value{#1}=3}{\mbox{$\gamma$}}{%
   \ifthenelse{\value{#1}=4}{\mbox{$\delta$}}{%
   \ifthenelse{\value{#1}=5}{\mbox{$\varepsilon$}}{%
   \ifthenelse{\value{#1}=6}{\mbox{$\zeta$}}{%
   \ifthenelse{\value{#1}=7}{\mbox{$\eta$}}{%
   \ifthenelse{\value{#1}=8}{\mbox{$\theta$}}{%
   \ifthenelse{\value{#1}=9}{\mbox{$\iota$}}{%
   \ifthenelse{\value{#1}=10}{\mbox{$\kappa$}}{%
   \ifthenelse{\value{#1}=11}{\mbox{$\lambda$}}{%
   \ifthenelse{\value{#1}=12}{\mbox{$\mu$}}{%
   \ifthenelse{\value{#1}=13}{\mbox{$\nu$}}{%
   \ifthenelse{\value{#1}=14}{\mbox{$\xi$}}{%
   \ifthenelse{\value{#1}=15}{\mbox{$\rm o$}}{%
   \ifthenelse{\value{#1}=16}{\mbox{$\pi$}}{%
   \ifthenelse{\value{#1}=17}{\mbox{$\varrho$}}{%
   \ifthenelse{\value{#1}=18}{\mbox{$\sigma$}}{%
   \ifthenelse{\value{#1}=19}{\mbox{$\tau$}}{%
   \ifthenelse{\value{#1}=20}{\mbox{$\upsilon$}}{%
   \ifthenelse{\value{#1}=21}{\mbox{$\varphi$}}{%
   \ifthenelse{\value{#1}=22}{\mbox{$\chi$}}{%
   \ifthenelse{\value{#1}=23}{\mbox{$\psi$}}{\mbox{$\omega$}%
  }}}}}}}}}}}}}}}}}}}}}}}}

\newcounter{subgreek}
{\end{list}}

\newcounter{subarabic}
{\end{list}}

\newcounter{subroman}
{\end{list}}

\newcount\skewfactor

\def\mathunderaccent#1#2 {\let\theaccent#1\skewfactor#2
\mathpalette\putaccentunder}
\def\putaccentunder#1#2{\oalign{$#1#2$\crcr\hidewidth
\vbox to.2ex{\hbox{$#1\skew\skewfactor\theaccent{}$}\vss}\hidewidth}}
\def\name{\mathunderaccent\tilde-3 }

\nc{\nname}{\name}

\nc{\even}{\ensuremath{\rm Even}}
\nc{\odd}{\ensuremath{\rm Odd}}

\nc{\al}{$\alpha$\  }
\nc{\om}{\omega}
\nc{\omm}{\ensuremath{\omega_1}}
\nc{\ep}{\varepsilon}
\nc{\tk}{\tilde{K}}
\nc{\concat}{\,\hat{} \,}   
\nc{\force}{\Vdash}
\nc{\fb}{f_{\overline{M}}}
\nc{\such}{\, : \,}
\newcommand{\la}{\langle}
\newcommand{\ra}{\rangle}

\nc{\meager}{\ensuremath{{\cal M}}}
\nc{\lebesgue}{\ensuremath{{\cal N}}}
\nc{\nulll}{\ensuremath{{\cal N}}}
\nc{\ksigma}{\ensuremath{{\bf K}_\sigma}}
\nc{\ideal}{\ensuremath{{\cal I}}}
\nc{\ga}{\ensuremath{\frak a}}
\nc{\AAA}{{\cal A}}   
\nc{\gc}{\ensuremath{\frak c}}
\nc{\gs}{\ensuremath{\frak s}}
\nc{\gh}{\ensuremath{\frak h}}
\nc{\gd}{\ensuremath{\frak d}}
\nc{\gb}{\ensuremath{\frak b}}
\nc{\gro}{\ensuremath{\frak g}}
\nc{\gu}{\ensuremath{\frak u}}
\nc{\gr}{\ensuremath{\frak r}}
\nc{\gt}{\ensuremath{\frak t}}
\nc{\fff}{\ensuremath{\frak f}}
\nc{\gm}{\ensuremath{\mathfrak{mcf}}}
\nc{\gge}{\ensuremath{\mathfrak e}}
\nc{\cfupro}{\ensuremath{\cf(\upro)}}
\nc{\cfvpro}{\ensuremath{\cf(\vpro)}}
\nc{\gp}{\ensuremath{\frak p}}
\nc{\gk}{\ensuremath{\frak k}}

\nc{\add}{\mbox{\ensuremath{{\rm add}}}}
\nc{\cov}[1]{\mbox{\ensuremath{{\rm cov}(#1)}}}
\nc{\unif}[1]{\mbox{\ensuremath{{\rm unif}(#1)}}}
\nc{\cof}[1]{{\mbox{\ensuremath{\rm cof}(#1)}}}

\nc{\addd}[2]{\mbox{\ensuremath{{\rm add}^{#1}(#2)}}}   
\nc{\covv}[2]{\mbox{\ensuremath{{\rm cov}^{#1}(#2)}}}   
\nc{\uniff}[2]{\mbox{\ensuremath{{\rm unif}^{#1}(#2)}}} 
\nc{\coff}[2]{{\mbox{\ensuremath{\rm cof}^{#1}(#2)}}}

\nc{\cd}{Cicho\'n's Diagram}

\nc{\MA}{\mbox{\sf MA}}
\nc{\PFA}{\mbox{\sf PFA}}
\nc{\OCA}{\mbox{\sf OCA}}
\nc{\GCH}{\mbox{\sf GCH}}
\nc{\CH}{\mbox{\sf CH}}
\nc{\zfc}{\mbox{\sf ZFC}}
\nc{\sch}{\mbox{\sf SCH}}
\nc{\ZF}{\mbox{\sf ZF}}
\nc{\NCF}{\mbox{\sf NCF}} 
\nc{\FD}{\mbox{\sf FD}}   
\nc{\SFT}{\mbox{\sf SFT}}
\nc{\fourG}{\mbox{\rm 4G}}
\nc{\fourI}{\mbox{\rm 4I}}
\nc{\past}{\ ;{\rm past}\;}
\nc{\Borelhood}{Borel measurability} 
\nc{\Pieinseins}{\mbox{${\bf \Pi}^1_1$}}
\nc{\seinseins}{\mbox{${\bf\Sigma}^1_1$}}
\nc{\seinszwei}{\mbox{${\bf\Sigma}^1_2$}}
\nc{\seinsdrei}{\mbox{${\bf\Sigma}^1_3$}}
\nc{\Deleinszwei}{\mbox{${\bf\Delta}^1_2$}}

\nc{\up}{\ensuremath{{\cal U}\mbox{\ensuremath{\rm -prod}}\,\omega}}
\nc{\upp}{\ensuremath{{\cal U}'\mbox{\ensuremath{\rm -prod}}\,\omega}}
\nc{\upro}{\ensuremath{{\cal U}\mbox{\ensuremath{\rm -prod}}\,\om}}
\nc{\fupro}{\ensuremath{f({\cal U})\mbox{\ensuremath{\rm -prod}}\,\om}}
\nc{\vpro}{\ensuremath{{\cal V}\mbox{\ensuremath{\rm -prod}}\,\om}}
\nc{\fpro}{\ensuremath{{\cal F}\mbox{\ensuremath{\rm -prod}}\,\om}}

\nc{\cff}[1]{{\text{cf}\,(#1)}}           
\nc{\cu}{\ensuremath{\cal U}}             
\nc{\ai}{\ensuremath{\forall^\infty}}     
\nc{\ei}{\ensuremath{\exists^\infty}}     
\nc{\ww}{\ensuremath{\omega^\omega}}      

\nc{\gw}{groupwise dense}

\nc{\kk}{car\-dinal cha\-rac\-teris\-tic}
\nc{\joker}{\ast}
\nc{\gtc}{Galois-Tukey connection} 
\nc{\ntc}{not nearly coherent} 
\nc{\nnc}{non-nearly coherent} 

\nc{\av}[1]{{\rm Av}_{#1}}
\nc{\eps}{\varepsilon}
\renewcommand{\epsilon}{\varepsilon}

\nc{\n}{{\bf n}}                 
\nc{\m}{{\bf m}}
\nc{\bT}{{\bf T}} 
\nc{\marginparr}[1]
{\marginpar{\makebox[4mm]{} {#1}}}
\nc{\footnoteee}{} 
\nc{\footnotee}{}  

\newcommand{\cal}{\mathcal}

\nc{\divs}{{c_0 \setminus \ell^1}}
\nc{\divser}{(\divs, \leq^*)/\thickapproy}
\nc{\bfin}{\RO(\pfin \setminus\{0\},\subseteq^*)}
\nc{\bdivser}{\RO(\divser)}
\nc{\inc}{{\rm INC}}
\nc{\com}{{\rm COM}}
\nc{\thickapproy}{\makebox{}\!\!\thickapprox}
\nc{\approy}{\makebox{}\!\!\approx}
\nc{\lessi}{\leqslant}
\nc{\gessi}{\geqslant}
\nc{\interior}[1]{{\rm int}(#1)}
\nc{\closure}[1]{{\rm cl}(#1)}
\nc{\Vo}{Vojt\'a\v{s}}

\nc{\precedeseq}{\leq^*} 
\nc{\precedes}{\prec}
\nc{\stronger}{\leqslant_{\bf P}}
\nc{\underlline}[1]{\hat{#1}}
\nc{\PO}{{\bf P}}
\nc{\charak}{\text{ch}}
\nc{\symom}{{\rm{Sym}(\omega)}}

\nc{\needed}{needed\ }
\nc{\neededc}{needed}
\nc{\Needed}{Needed\ }
\nc{\wneeded}{weakly needed\ }
\nc{\Wneeded}{Weakly needed\ }
\nc{\wneededc}{weakly needed}

\nco{\last}{last}
\nc{\mup}{m_{\rm up}}
\nc{\mdn}{m_{\rm dn}}
\nco{\may}{may}
\nco{\aver}{av} 
\nco{\norm}{nor} 
\nco{\val}{val} 
\nco{\dis}{dis} 
\nco{\basis}{basis}
\nco{\pos}{pos}
\nco{\spec}{spec}
\nc{\err}{\mbox{err}}
\nc{\eee}{\mbox{e}}
\nco{\Expect}{Exp}
\nco{\rt}{rt}
\nco{\pr}{pr}
\nco{\suc}{suc}
\nco{\splitt}{spl}
\nco{\halv}{h}
\nco{\Add}{Add}
\nco{\Cov}{Cov}
\nco{\Unif}{Unif}
\nco{\Cof}{Cof}
\nco{\htt}{ht}
\nco{\cl}{cl}
\nco{\Levy}{Levy}
\nco{\set}{set}
\nc{\bbforcing}{\mathbb A}
\nc{\itername}{\mathfrak q}
\nc{\iterp}{\mathfrak p}
\nc{\iterq}{\mathfrak q}
\nc{\invcm}{\rm inv_{cm}}
\nc{\invcf}{\rm inv_{cf}}
\nc{\invgm}{\rm inv_{gm}}

\nc{\Q}{\mathbb Q}
\nc{\subsetsim}{\underset{\raise0.6em\hbox{$\sim$}}{\subset}}
\newcommand{\subsim}{\underset{\raise20pt\hbox{$\rightarrow$}}{\rightarrow}}
\newcommand{\ssim}{\overset{\raise-40pt\hbox{$\leftarrow$}}{\subsim}}
\nc{\rest}{\restriction}
\nc{\bF}{\mathbb F}\nc{\F}{\mathbb F}
\nc{\R}{\mathbb R}
\nc{\M}{\mathbb M}
\nc{\bQ}{\mathbb Q}
\nc{\bP}{\mathbb P}
\nc{\bG}{\mathbf G}

\nc{\cJ}{\mathcal J}

\nc{\cB}{\mathcal B}
\nc{\bV}{\mathbf{V}}
\nc{\cE}{\mathcal E}
\nc{\cP}{\mathcal P}
\nc{\cU}{\mathcal U}
\nc{\cV}{\mathcal V}
\nc{\cW}{\mathcal W}
\nc{\cC}{\mathcal C}
\nc{\cT}{\mathcal T}
\nc{\cX}{\mathcal X}
\nc{\cD}{\mathcal D}
\nc{\cS}{\mathcal S}
\nc{\cA}{\mathcal A}
\nc{\cG}{\mathcal G}
\nc{\cI}{\mathcal I}
\nc{\ba}{\mathbf{a}}
\nc{\bC}{\mathbb C}
\begin{document}
\title{A version of $\kappa$-Miller forcing}
\thanks{
This research, no.~1154 on the second author's list, was partially supported by European Research Council grant 338821.}
\author{Heike Mildenberger and Saharon Shelah}
\address{Heike Mildenberger, Albert-Ludwigs-Universit\"at Freiburg,
Mathematisches Institut, Abteilung f\"ur math. Logik, Ernst--Zermelo--Stra\ss e~1, 
79104 Freiburg im Breisgau, Germany}
\address{Saharon Shelah, Institute of Mathematics, The Hebrew University of Jerusalem,
Edmond Safra Campus Givat Ram, 9190401 Jerusalem, Israel}

\date{
  March 5, 2019}
\begin{abstract}
  Let $\kappa$ be an  uncountable cardinal such that
  $2^{<\kappa} = \kappa$ or just
  $\cf(\kappa) > \omega$,  $2^{2^{<\kappa}}= 2^\kappa$, and
  $([\kappa]^\kappa, \supseteq)$ collapses $2^\kappa$
  to $\omega$. We show under these assumptions
the 
$\kappa$-Miller forcing with club many splitting nodes collapses
$2^\kappa$ to $\omega$ and adds a $\kappa$-Cohen real.
\end{abstract}

\maketitle

\section{Introduction}
\label{S1}

Many of the tree forcings on the classical Baire space have
various analogues for higher cardinals. Here we are concerned with 
Miller forcing \cite{miller:rational}. For a $\kappa$-version
of Miller forcing, in addition
to superperfectness one usually requires
(see, e.g., \cite[Section~5.2]{BrendleBrooke-TaylorFriedmanMontoya}) limits of
length $< \kappa$ of splitting
nodes be splitting nodes as well and that splitting mean
splitting into a club.
In this paper we investigate a version of $\kappa$-Miller forcing where
this latter requirement is waived. We show: If $\cf(\kappa) > \omega$,
$\cf(\kappa) = \kappa$ or $\cf(\kappa) < 2^{\cf(\kappa)} \leq \kappa$,
$2^{2^{<\kappa}} = 2^\kappa$, and
there is a $\kappa$-mad
family of size $2^\kappa$, then this variant of Miller forcing
is related to the forcing $([\kappa]^\kappa, \supseteq)$ and
collapses $2^\kappa$ to $\omega$.
In particular, if $\omega<\kappa^{<\kappa} = \kappa$, then our four premises
are fulfilled.

Throughout the paper we let
$\kappa$ be an uncountable cardinal.
We write $\trianglelefteq$ for end extension of functions whose domains
are ordinals. If $\dom(t), i $ are ordinals, we write $t \concat \la i\ra$ for 
the concatenation of $t$ with the singleton function 
$\{(0,i)\}$, i.e.,  $t \concat \la i\ra =t \cup \{(\dom(t), i)\}$.
We denote forcing orders in the form $(\bP,\leq_{\bP})$
and let $p \leq_\bP q$ mean that $q$ ist \emph{stronger} than $p$.
We write ${}^{\lambda>} \kappa$ for the set of functions $f \colon \alpha \to \kappa$ for some $\alpha< \lambda$. The domain $\alpha$ of $f$ is also called
the length of $f$. The set of subsets of $\kappa$ of size $\kappa$ is denoted by $[\kappa]^\kappa$.

\begin{definition}\label{1.1}
\begin{myrules}
\item[(1)] $\bQ_\kappa^1$ is the forcing $([\kappa]^\kappa, \supseteq)$.
\item[(2)] $\bQ_\kappa^2$ is the following version of $\kappa$-Miller
forcing: Conditions are trees $T \subseteq {}^{\kappa > } \kappa$ that are
$\kappa$ \emph{superperfect}: for each $s \in T$ there is $s \trianglelefteq  t$
such that $t$ is a $\kappa$-splitting node of $T$ (short $t \in \splitt(T)$).
A node $t \in T$ is called a \emph{$\kappa$-splitting node} if 
\[
\set_p(t) = \{ i < \kappa \such t \concat \la i \ra \in T\}
\]
  has size $\kappa$.
  We furthermore require that the limit of an increasing in the tree order sequence of length less than $\kappa$ of $\kappa$-splitting nodes  is a $\kappa$-splitting node
  if it has length less than $\kappa$.

  For $p,q \in \bQ_\kappa^2$ we write $p \leq_{\bQ^2_\kappa} q$ if $q \subseteq p$.
  So subtrees are stronger conditions.
  
\item[(3)] For 
  $p \in \bQ_\kappa^2$ and $\eta \in p$ we let $\suc_p(\eta) = 
\{\eta' \in {}^{\kappa>} \kappa \such (\exists i \in \kappa) (\eta' = \eta \concat \la i \ra \in p)\}$.
\nothing{
\item [(4)] 
For $p \in \bQ_\kappa^2$ and $\eta \in p$ we let $\set_p(\eta) = 
\{i \in \kappa \such \eta \concat \la i \ra \in p\}$.
\item[(5)] For $p \in \bQ_\kappa^2$ we let $\splitt(p) = \{\eta \in 
  p \such |\suc_p(\eta)| =\kappa\}$ be the set of splitting nodes of $p$.
 }
\item[(4)] Let $\eta \in p \in \bQ^2_\kappa$. We let $p^{\la \eta \ra} = \{\nu\in p \such \nu \trianglelefteq \eta \vee \eta \trianglelefteq \nu\}$.
\item[(5)] For $a, b \subseteq \kappa$ we write $a \subseteq^*_\kappa b$ if
  $|a \setminus b|< \kappa$.
  \nothing{
\item[(add)] Let $\theta \leq \kappa$.
For a subtree $\tau \subseteq {}^\theta \kappa$ we let
$\lim_\theta(\tau) = |\{\eta \in {}^\theta \kappa \such (\forall i < \theta)
\eta \rest i \in \tau \}|$. 
$f(\kappa, \theta, \mu) = \sup\{\lim_\theta(\tau) \such 
\tau \subseteq {}^\theta \kappa \wedge |\tau| \leq \mu\}$.

\item[(5)] If $t \in {}^{\kappa >} \kappa$ we let $\last(t)$ be undefined
  if $\dom(t)$ is a limit ordinal and be $t(\dom(t)-1)$ otherwise.
}
\end{myrules}
\end{definition}

Each of the two forcing orders $\bP$  has a weakest element, denoted by $0_{\bP}$. Namely, $\bQ^1_\kappa$ has as a weakest element $0_{\bQ^1_\kappa} = \kappa$, and $\bQ^2_\kappa$ has as a weakest element the full tree ${}^{\kappa>} \kappa$. 
We write $\bP\Vdash \varphi$ if the weakest condition $0_\bP$ forces $\varphi$.

\section{Results about $\bQ^1_\kappa$}

We will apply the following result for $\chi = 2^\kappa$.

\begin{theorem}  \label{861}
  (\cite[Theorem 0.5]{Sh:861})
  \begin{myrules}
    \item[(1)]
Under the assumption of an antichain of size $\chi$ in $\bQ^1_\kappa$,
$\bQ^1_\kappa$ collapses $\chi$ to $\aleph_0$ if
$\aleph_0 < \cf(\kappa) = \kappa$ or if
$\aleph_0 < \cf(\kappa) < 2^{\cf(\kappa)} \leq  \kappa$.
\item[(2)]
Under the assumption of an antichain of size $\chi$ in $\bQ^1_\kappa$,
$\bQ^1_\kappa$ collapses $\chi$ to $\aleph_1$ in the case of
$\aleph_0 = \cf(\kappa)$.
  \end{myrules}
\end{theorem}

\begin{definition} A family $\cA \subseteq [\kappa]^\kappa$ is called a
  $\kappa$-almost disjoint family if for $A \neq B \in \cA$, $|A \cap B|<\kappa$. A $\kappa$-almost disjoint family of size at least $\kappa$ that is maximal
  is called a $\kappa$-mad family.
\end{definition}

\begin{observation}\label{obs}
If $2^{<\kappa} = \kappa$, there is a $\kappa$-mad family
$\cA\subseteq[\kappa]^\kappa$ of size $2^\kappa$.
\end{observation}

\begin{proof} We let $f \colon {}^{\kappa>} 2 \to \kappa$ be an injection.
  We assign to each branch  $b$ of ${}^{\kappa>}2$ a set 
$a_b=\{f(s) \such s \in b\}$. Then we  complete the resulting family 
$\{a_b \such b \mbox{ branch of }{}^{\kappa>} 2\}$ to a maximal $\kappa$-almost disjoint family.
\end{proof}

\begin{observation}\label{obs2}
  If $\bQ^1_\kappa$ collapses $2^\kappa$ to $\omega$, then there is
a $\kappa$-mad family
$\cA$ of size $2^\kappa$.
\end{observation}

\begin{proof} $\bQ^1_\kappa$ cannot have the $2^\kappa$-c.c. Hence there is an antichain of size $2^\kappa$. This is a $\kappa$-ad family, and we extend it to a $\kappa$-mad family. \end{proof}

For further use, we indicate the hypothesis for each technical step.

\begin{lemma}\label{inserted}
  Suppose that $\bQ_\kappa^1$ collapses $2^\kappa$ to $\omega$. Then
  there is a $\bQ^1_\kappa$-name $\name{\tau} \colon \aleph_0 \to 2^\kappa$
  for a surjection, and there is a labelled tree ${\mathcal T}=
  \la (a_\eta, n_\eta, \varrho_\eta) \such \eta \in {}^{\omega>}(2^\kappa)\ra$ with the following properties
  \begin{myrules}
   \item[(a)] $a_{\la \ra} = \kappa$ and for any $\eta \in {}^{\omega>}(2^\kappa)$, $a_\eta \in [\kappa]^\kappa$.
   \item[(b)] $\eta_1 \triangleleft \eta_2$ implies $a_{\eta_1} \supseteq a_{\eta_2}$.
     \item[(c)] $n_\eta \in [\lg(\eta)+1, \omega)$.
     \item[(d)] If $a \in [\kappa]^\kappa$ then there is some $\eta \in {}^{\omega >} (2^\kappa)$ such that $a \supseteq a_\eta$.
             \item[(e)]
          If
          $\eta\concat \la \beta \ra \in T$ then $a_{\eta \concat \la \beta\ra}$
          forces $\name{\tau} \rest
          n_\eta = \varrho_{\eta \concat \la \beta \ra}$ for some
          $\varrho_{\eta \concat \la \beta \ra} \in {}^{n_\eta} (2^\kappa)$,
          such that the $\varrho_{\eta \concat \la \beta \ra}$, $\beta \in 2^\kappa$, are pairwise different.  Hence for any $\eta \in {}^{\omega>}(2^\kappa)$, the family $\{a_{\eta \concat\la \alpha \ra} \such \alpha <  2^\kappa \}$ is a $\kappa$-ad family in $[a_\eta]^\kappa$.
      \nothing{\item[(f)]
        \[
          {\rm Pos}(a_\eta, n_\eta)= \{\varrho \in {}^{n_{\eta}} (2^\kappa) \such a_\eta \not\Vdash_{\bQ^1_\kappa} \name{\tau} \rest n_\eta \neq \varrho\}
        \] has cardinality $2^\kappa$.}
\end{myrules}
\end{lemma}

\proof
Let $\name{\tau}$ be a $\bQ^1_\kappa$-name such that
$\bQ^1_\kappa \Vdash \name{\tau} \colon \aleph_0 \to 2^\kappa$ is onto.
For $\alpha < 2^\kappa$ let $AP_\alpha$ be the set of objects $\bar{m}$ satisfying
\begin{myrules} 
\item[$(\ast)_1$]
  \begin{myrules} 
  \item[(1.1)] $\bar{m} = (T,\bar{a},\bar{n},\bar{\varrho}) = (T_{\bar{m}}, \bar{a}_{\bar{m}}, \bar{n}_{\bar{m}}, \bar{\varrho}_{\bar{m}})$.
  \item[(1.2)] $T$ is a subtree of $({}^{\omega>} (2^\kappa), \triangleleft)$ of cardinality $\leq |\alpha| + \kappa$ and $\la \ra \in T$.
  \item[(1.3)] 
    $\bar{a} = \la a_\eta \such \eta \in T\ra$ fulfils
    $\eta \triangleleft \nu \rightarrow a_\nu \subseteq a_\eta$ and
    $ a_{\la \ra} = \kappa$ and $a_\eta \in[\kappa]^{\kappa}$.
    \item[(1.4)]
    $\bar{n} = \la n_\eta \such \eta \in T \ra$ fulfils  $\dom(\varrho_{\eta \concat \la \beta \ra}) = n_\eta > \lg(\eta)$ for any $\eta \concat \la \beta \ra \in T$.
      
\item[(1.5)]  If
          $\eta\concat \la \beta \ra \in T$, then $a_{\eta \concat \la \beta\ra}$
          forces a value to $\name{\tau} \rest
          n_\eta$ called $\varrho_{\eta \concat \la \beta \ra}$ and for $\beta \neq \gamma$
          we have $\varrho_{\eta \concat \la \beta \ra} \neq \varrho_{\eta \concat
            \la\gamma\ra}$.
          Hence for any $\eta \concat\la \beta\ra$, $\eta \concat \la \gamma\ra \in T_{\bar{m}}$, 
    $\beta \neq \gamma$ implies $a_{\eta \concat\la \beta \ra } \cap a_{\eta \concat \la \gamma \ra} \in [\kappa]^{<\kappa}$.
     \item[(1.6)] For $\eta \in T_{\bar{m}}$, we let
        \[
          {\rm Pos}(a_\eta, n_\eta)= \{\varrho \in {}^{n_{\eta}} (2^\kappa) \such a_\eta \not\Vdash_{\bQ^1_\kappa} \name{\tau} \rest n_\eta \neq \varrho\},
          \]
          and require that the latter has cardinality $2^\kappa$.
\end{myrules}
  In the next items we state some properties of $AP_\alpha$ that are derived
  from $(\ast)_1$.
\item[$(\ast)_2$]
    $AP = \bigcup \{AP_\alpha \such \alpha< 2^\kappa\}$
  is ordered naturally by $\leq_{AP}$, which means end extension.
\item[$(\ast)_3$]
  \begin{myrules}
    \item[(a)]
    $AP_\alpha$ is not empty and increasing in $\alpha$.
  \item[(b)] For infinite $\alpha $, 
    $AP_\alpha$ is closed under unions of increasing sequences of length $< |\alpha|^+$.
  \end{myrules}

\item[$(\ast)_4$] Let $\gamma < 2^\kappa$.
  If $\bar{m} \in AP_\gamma$ and $\eta \in T_{\bar{m}}$ and $\eta \concat \la \alpha \ra \not\in T_{\bar{m}}$ then there is $\bar{m}' \in AP_\gamma$ such that $\bar{m} \leq_{AP} \bar{m}'$ and $T_{\bar{m}'} = T_{\bar{m}} \cup \{\eta \concat \la \alpha \ra\}$.

Proof: For $\eta \in T_{\bar{m}}$,
  \[\cU={\rm Pos}(a_\eta, n_\eta)= \{\varrho \in {}^{n_{\eta}} (2^\kappa) \such a_\eta \not\Vdash_{\bQ^1_\kappa} \name{\tau} \rest n_\eta \neq \varrho\} \mbox{ has size } 2^\kappa,\]
  whereas
\[\Lambda_\eta = \{ \varrho_{\eta \concat \la \beta \ra} \rest n_\eta \such \beta \in 2^\kappa
\wedge \eta \concat \la \beta \ra \in T_{\bar{m}}\}\]
is of size $\leq |T_{\bar{m}}| \leq |\gamma | + \kappa$.
 Hence we can choose $\varrho_\ast \in  \cU \setminus \Lambda_\eta$
 and $b_\ast \in [a_\eta]^\kappa$ such that $b_\ast \Vdash_{\bQ^1_\kappa} \varrho_\ast = \name{\tau} \rest n_\eta$. We let $\varrho_{\eta \concat \la \alpha \ra} = \varrho_\ast$.  Since $b_\ast$ forces a value of $\tau \rest n_\eta$ that is incompatible with the one forced by $a_{\eta \concat \la \beta \ra}$ for any $\eta \concat \la \beta \ra \in T_{\bar{m}}$, the set $b_\ast$ is $\kappa$-almost disjoint from $ a_{\eta \concat \la \beta \ra}$
  for any $\eta \concat \la \beta \ra \in T_{\bar{m}}$.
 We take $b_\ast = a_{\bar{m}', \eta \concat \la \alpha \ra} \subseteq a_{\bar{m},\eta}$.

 Since $\cf(2^\kappa ) > \aleph_0$ and since
 \[
 |\{\rge(\varrho) \such \varrho \in  {}^{\omega >} (2^\kappa) \wedge b_\ast \not\Vdash_{\bQ^1_\kappa} \name{\tau} \rest n \neq \varrho\}| = 2^\kappa,
 \]
        there is  an $n$ such that
\[
{\rm Pos}(b_\ast, n)= \{\varrho \in {}^{n} (2^\kappa) \such b_\ast \not\Vdash_{\bQ^1_\kappa} \name{\tau} \rest n \neq \varrho\}
\]
has cardinality $2^\kappa$.
          We take the minimal one and let it be $n_{\eta \concat \la \alpha \ra}$.
\item[$(\ast)_5$]
  If $\bar{m} \in AP_\alpha$ and $a \in [\kappa]^\kappa$ then there is some $\bar{m}' \geq \bar{m}$, such that there is $\eta \in T_{\bar{m}'}$
  with $a_{\bar{m}',\eta} \subseteq a$.

  Let
  \[\cU_a = \{ \varrho \in {}^{\omega>} (2^\kappa) \such
  a \not\Vdash_{\bQ^1_\kappa} \varrho \! \not\triangleleft \, \name{\tau}\},
  \]
  i.e.
  \[\cU_a = \{ \varrho \in {}^{\omega>} (2^\kappa) \such
  (\exists b \geq_{\bQ^1_\kappa}  a)(b \Vdash_{\bQ^1_\kappa} \varrho \triangleleft \name{\tau})\}.
  \]
  This set has cardinality $2^\kappa$  because
  $\bQ^1_\kappa \Vdash \name{\tau}  \colon \omega \to 2^\kappa$ is onto.
  We take $n$ minimal such that
  \[\cU_{a,n} = \{ \varrho \in {}^{n} (2^\kappa) \such
  (\exists b \geq_{\bQ^1_\kappa}  a)(b \Vdash_{\bQ^1_\kappa} \varrho \triangleleft \name{\tau})\}
  \]
  has size $2^\kappa$. We let
  \[\set_n^+(\bar{m}) = \{ \varrho_{\eta} \such \eta \in T_{\bar{m}}, \lg(\varrho_\eta) \geq n\}.
  \]
Clearly $|\set_n^+(\bar{m})| \leq
  |T_{\bar{m}}| \leq |\gamma | + \kappa$.
   Thus we can take $\varrho_a \in \cU_{a,n}$ that is incompatible
    with every element of $\set^+_n(\bar{m})$. We take some $b_a \in [a]^\kappa$ such that
    $b_a \Vdash_{\bQ^1_\kappa} \varrho_a \trianglelefteq \name{\tau}$.
    The set
    \[\Lambda_a = \{ \eta \in T_{\bar{m}} \such
      b_a \subseteq^*_\kappa a_\eta\}\]
      is $\triangleleft$-linearly ordered by $(\ast)_1$ clauses 1.3 and 1.5 and $\la \ra \in \Lambda_a$. Since $b_a$ does not pin down $\name{\tau}$,
      $\Lambda_a$  has a $\triangleleft$-maximal member $\eta_a$.
      Now we take $\alpha_\ast = \min\{\beta \such \eta_a \concat \la \beta \ra \not\in T_{\bar{m}}\}$.
      For any $\eta_a \concat \la \beta \ra \in T_{\bar{m}}$ we have
      $\varrho_{\eta_a \concat \la \beta \ra} $ and $\varrho_a$ are incompatible,
      and hence $a_{\eta_a \concat \la \beta \ra} \cap b_a \in [\kappa]^{<\kappa}$.
Now we choose $b_a^1 \in [b_a]^\kappa$ and $\varrho^*_a$ such that
     $ b_a^1 \Vdash_{\bQ^1_\kappa} \varrho^\ast_a \triangleleft \name{\tau}$ and $\lg(\varrho^\ast_a) \geq n_{\bar{m},\eta_a}> \lg(\eta_a)$.
 
We let \begin{eqnarray*}
  T_{\bar{m}'} &=& T_{\bar{m}} \cup \{\eta_a \concat \la \alpha_\ast \ra\},\\
  a_{\eta_a \concat \la \alpha_\ast \ra} &=& b^1_a,
  \end{eqnarray*}
We let $n_{\eta_a \concat \la \alpha_\ast\ra}$ be the minimal $n$ such that $|{\rm Pos}(b_a^1,n)| \geq 2^\kappa$.
So $(\ast)_5$ holds.
\end{myrules}

Now we are ready to construct $\cT$ as in the statement of the lemma.
We do this by recursion on $\alpha \leq 2^{\kappa}$.
First we enumerate $[\kappa]^\kappa$ as  $\la c_\alpha \such \alpha< 2^\kappa\ra$,
and we enumerate
${}^{\omega >} (2^\kappa)$ as $\la\eta_\alpha\such \alpha < 2^\kappa\ra$ such that
$\eta_\alpha \triangleleft \eta_\beta$ implies $\alpha < \beta$. We choose
an increasing sequence
$\bar{m}_\alpha$ by induction on $\alpha<2^\kappa$. We start with the tree
$\{\la \ra\}$, $a_{\la \ra}= \kappa$, $\varrho_{\la \ra} = \emptyset$, $n_{\la \ra}$
be minimal such that $|{\rm Pos}(\kappa,n)| = 2^\kappa$. In the odd successor steps we take $\bar{m}_{2\alpha +1} \geq_{AP} \bar{m}_\alpha$ so that
$a_\eta \subseteq c_\alpha$ for some $\eta \in T_{2\alpha+1}$. This is done according to $(\ast)_5$.
In the even successor steps we take $\bar{m}_{2\alpha+2} \geq_{AP} \bar{m}_{2\alpha+1}$ such that $\eta_\alpha \in T_{2\alpha+2}$. Since
all initial segments of $\eta_\alpha$ appeared among the $\eta_\beta$, $\beta < \alpha$, $\bar{m}_{2\alpha+2}$ is found 
according to $(\ast)_4$.
In the limit steps we take unions.
Then ${\mathcal T}$ that is given by the the last
three components of $\bar{m}_{2^\kappa}$ has properties (a) to (e).
\proofend

Since $\tau= \name{\tau}[G]$ is not in $\bV$, for any $\cT$ as in Lemma~\ref{inserted}
no sequence of first components of a branch, i.e., no $\la a_{f\rest n} \such n \in \omega\ra$, $f \in {}^\omega(2^\kappa) \cap {\bV}$, has a $\subseteq^*_\kappa$-lower bound.

\section{Transfer to $\bQ^2_\kappa$}
\label{S3}

In this section we use the tree $\cT$ from Lemma~\ref{inserted}
for finding $\bQ^2_\kappa$-names.

\begin{definition} Let $\mu, \lambda$ be cardinals. For $\nu, \nu' \in {}^{\lambda >}\mu$ we write
  $\nu \perp \nu'$ if
  $\nu \not\trianglelefteq \nu'$ and  $\nu' \not\trianglelefteq \nu$.
\end{definition}

Typical pairs $(\lambda, \mu)$ are $(\omega, 2^\kappa)$ and $(\kappa,\kappa)$.

An important tool for the analysis of $\bQ^2_\kappa$ is
the following particular kind of fusion sequence
$\la p_\alpha \such \alpha < \kappa^{<\kappa} \ra$
in $\bQ^2_\kappa$. Since we do not suppose $\kappa^{<\kappa} = \kappa$,
a fusion sequence can be longer than $\kappa$. An important property
is that for each
$\nu \in {}^{\kappa>}\kappa$ there is at most one $\alpha< \kappa^{<\kappa}$
such that $\set_{p_\alpha}(\nu) \supsetneq \set_{p_{\alpha+1}}(\nu)$.

\begin{lemma} \label{fusionstark}
  Let $\la \nu_\alpha \such \alpha < \kappa^{<\kappa} \ra$ be an
  injective enumeration of $\kappa^{<\kappa}$ such that
  \begin{equation} \label{gut}
    \nu_\alpha \triangleleft \nu_\beta \rightarrow \alpha < \beta.
    \end{equation}
    Let $\la p_\alpha, \nu_\alpha, c_\alpha \such \alpha < \kappa^{<\kappa} \ra$ be a sequence
    such that for any $\alpha \leq \lambda$ the following holds:
    \begin{myrules}
      \item[(a)] $p _0 \in \bQ^2_\kappa$.
    \item[(b1)] If $\alpha = \beta +1 < \kappa^{<\kappa}$ 
      and $\nu_\beta \in sp(p_\beta)$, then 
      \begin{align*}
\begin{split}
   c_\beta & \in [\suc_{p_\beta}(\nu_\beta)]^\kappa \mbox{ and }\\
  p_{\alpha}= p_\beta(\nu_\beta, c_\beta)  &:= \bigcup\{p_\beta^{\la \nu_\beta \concat \la i  \ra \ra} \such i \in c_\beta \}
      \cup \bigcup\{p_\beta^{\la \eta \ra} \such \eta \not\trianglelefteq \nu_\beta
      \wedge \nu_\beta \not\trianglelefteq \eta\}
\end{split}
      \end{align*}
\item[(b2)]
If $\alpha = \beta +1 < \kappa^{<\kappa}$ 
and $\nu_\beta \not\in \splitt(p_\beta)$ then $p_\alpha= p_\beta$.
  \item[(c)] $p_\alpha = \bigcap \{p_\beta \such \beta<\alpha \}$ for limit $\alpha \leq \kappa^{<\kappa}$.
    \end{myrules}
    Then for any $\lambda \leq \kappa^{<\kappa}$, $p_\lambda \in \bQ^2_\kappa$ and $\forall \beta < \lambda$, $p_\beta \leq_{\bQ^2_\kappa} p_\lambda$.
  \end{lemma}
  
\begin{proof} We go by induction on $\lambda$. The case $\lambda =0$ and the successor steps are obvious. So we assume that $\lambda \leq \kappa^{<\kappa}$ is a limit ordinal and
$p_\alpha \in \bQ^2_\kappa$ for $\alpha < \lambda$. Since $\emptyset \in p_\lambda$, $p_\lambda$ is not empty, and $p_\lambda$ clearly is a tree.
  Let $t \in p_\lambda$. We show that there is $t' \trianglerighteq t$ that is a splitting node in $p_\lambda$.
\nothing{
  First we state some general consideration: Suppose we are at stage $\beta < \kappa^{<\kappa}$ in the inductive construction. 
The choice of the $p_\gamma$, $\gamma \leq \beta$,  equations \eqref{gut} and \eqref{ast1} together imply that
    \[\nu_\beta \in \splitt(p_\beta) \setminus \{ \nu \such (\exists \gamma < \beta)
    (\nu \trianglelefteq \nu_\gamma)\}.\]
Moreover, Equation
   in fusion sequence definition  yields
    $\suc_{p_{\beta+1}}(\nu_\beta) = c_\beta$, so $\nu_\beta \in \splitt(p_{\beta+1})$. Hence for any $\nu_\beta \in \splitt(p_\beta) \setminus \{\nu_\gamma \such \gamma < \beta\}$ we have for any $\delta \in [\beta+1, \kappa^{<\kappa}]$:
    \begin{equation} \label{ast3}
      \suc_{p_{\beta+1}}(\nu_\beta) = \suc_{p_\delta}(\nu_\beta) \wedge \nu_\beta \in \splitt(p_\delta).
    \end{equation}
}
We fix the smallest $\alpha$ such that $\nu_\alpha \trianglerighteq_{p_0} t$ is a splitting  node
in $p_0$. Then in $p_0$ there are no splitting nodes
in $\{s \such t \trianglelefteq s \triangleleft \nu_\alpha \}$.
Hence $\nu_\alpha \in \splitt(p_{\beta})$ for any $\beta \in [0,\lambda]$.
      
Now we show that
the limit of splitting nodes in $p_\lambda$ is a splitting node.
Let $\gamma < \lambda$ and let $\la \nu^{i} \such i < \gamma \ra$ be an $\triangleleft$-increasing sequence of splitting nodes of $p_\lambda$ with union
$\nu \in \kappa^{<\kappa}$.
Then $\nu$ is a splitting node of each $p_\alpha$, $\alpha < \lambda$,
and also in $p_\lambda$ since $\la \set_{p_\alpha}(\nu) \such \alpha < \lambda\ra$
has at most two entries and their intersection has size $\kappa$.
\end{proof}

We need yet another type of fusion sequence.

\begin{definition} \label{front}
  Let $p \in \bQ^2_\kappa$ and let $\nu \in \splitt(p)$.
  \begin{myrules}
  \item[(1)]
    Let
  $i \in \set_p(\nu)$.
  We say $\eta$  is \emph{the shortest
  splitting node above $\nu \concat \la i \ra$ in $p$}
  and write $\eta = \nex_p(\nu \concat i)$  if
  $\eta$ is the shortest splitting point in $p$ such that $\eta \supseteq \nu \concat\la i \ra$. Equality is allowed.
  \item[(2)]
    We say $F \subseteq p$ is \emph{the front of next splitting nodes
      above $\nu$ in $p$}, if
  \[ F = \{ \eta' \in \splitt(p) \such
  \exists (\eta \in \suc_p(\nu))(\eta' = \nex_p(\eta))\}.
  \]
  \end{myrules}
  \end{definition}

\begin{lemma} \label{fusionstark2}
  Let $\la \nu_\alpha \such \alpha < \kappa^{<\kappa} \ra$ be an
  injective enumeration of $\kappa^{<\kappa}$ such that
  \begin{equation} \label{gut}
    \nu_\alpha \triangleleft \nu_\beta \rightarrow \alpha < \beta.
    \end{equation}
    Let $\la p_\alpha, \nu_\alpha, c_\alpha, F_{\alpha} \such \alpha < \kappa^{<\kappa} \ra$ be a sequence
    such that for any $\alpha \leq \lambda$ the following holds:
    \begin{myrules}
      \item[(a)] $p _0 \in \bQ^2_\kappa$.
    \item[(b1)] If $\alpha = \beta +1 < \kappa^{<\kappa}$ 
      and $\nu_\beta \in sp(p_\beta)$, then 
$c_\beta  \in [\suc_{p_\beta}(\nu_\beta)]^\kappa$,
$F_{\beta}$ contains for each $i \in c_\beta$ exactly one
$\eta \in \splitt(p_\beta^{\la \nu_\beta \concat \la i \ra \ra})$, and
      \begin{equation*}
        \begin{split}
          p_{\alpha}= p_\beta(\nu_\beta, c_\beta, F_{\beta}) := & \bigcup\{p_\beta^{\la \eta \ra} \such i \in c_\beta, \eta \in F_{\beta} \}\\
  &
      \cup \bigcup\{p_\beta^{\la \eta \ra} \such \eta \not\trianglelefteq \nu_\beta
      \wedge \nu_\beta \not\trianglelefteq \eta\}.
\end{split}
      \end{equation*}
      Note that this implies
     that $F_\beta$ is the front of next splitting nodes of $p_\alpha$ above $\nu_\beta$.
      \item[(b2)]
If $\alpha = \beta +1 < \kappa^{<\kappa}$ 
and $\nu_\beta \not\in \splitt(p_\beta)$ then $p_\alpha= p_\beta$.
  \item[(c)] $p_\alpha = \bigcap \{p_\beta \such \beta<\alpha \}$ for limit $\alpha \leq \kappa^{<\kappa}$.
    \end{myrules}
    Then for any $\lambda \leq \kappa^{<\kappa}$, $p_\lambda \in \bQ^2_\kappa$ and $\forall \beta < \lambda$, $p_\beta \leq_{\bQ^2_\kappa} p_\lambda$.
  \end{lemma}
  
\begin{proof} We go by induction on $\lambda$. The case $\lambda =0$ and the successor steps are obvious. So we assume that $\lambda \leq \kappa^{<\kappa}$ is a limit ordinal and
$p_\alpha \in \bQ^2_\kappa$ for $\alpha < \lambda$. Since $\emptyset \in p_\lambda$, $p_\lambda$ is not empty, and $p_\lambda$ clearly is a tree.
  Let $t \in p_\lambda$. We show that there is $t' \trianglerighteq t$ that is a splitting node in $p_\lambda$.
\nothing{
  First we state some general consideration: Suppose we are at stage $\beta < \kappa^{<\kappa}$ in the inductive construction. 
The choice of the $p_\gamma$, $\gamma \leq \beta$,  equations \eqref{gut} and \eqref{ast1} together imply that
    \[\nu_\beta \in \splitt(p_\beta) \setminus \{ \nu \such (\exists \gamma < \beta)
    (\nu \trianglelefteq \nu_\gamma)\}.\]
Moreover, Equation
    \eqref{ast2}  yields
    $\suc_{p_{\beta+1}}(\nu_\beta) = c_\beta$, so $\nu_\beta \in \splitt(p_{\beta+1})$. Hence for any $\nu_\beta \in \splitt(p_\beta) \setminus \{\nu_\gamma \such \gamma < \beta\}$ we have for any $\delta \in [\beta+1, \kappa^{<\kappa}]$:
    \begin{equation} \label{ast3}
      \suc_{p_{\beta+1}}(\nu_\beta) = \suc_{p_\delta}(\nu_\beta) \wedge \nu_\beta \in \splitt(p_\delta).
    \end{equation}
}
We fix the smallest $\alpha$ such that $\nu_\alpha \trianglerighteq_{p_0} t$ is a splitting  node
in $p_0$. Then in $p_0$ there are no splitting nodes
in $\{s \such t \trianglelefteq s \triangleleft \nu_\alpha \}$.
Hence $\nu_\alpha \in \splitt(p_{\beta})$ for any $\beta \in [0,\lambda]$.
      
Now we show that
the limit of splitting nodes in $p_\lambda$ is a splitting node.
Let $\gamma < \lambda$ and let  $\la \nu^{i} \such i < \gamma \ra$ be an $\triangleleft$-increasing sequence of splitting nodes of $p_\lambda$ with union
$\nu \in \kappa^{<\kappa}$.
Then $\nu$ is a splitting node of each $p_\alpha$, $\alpha < \lambda$,
and also in $p_\lambda$ since $\la \set_{p_\alpha}(\nu) \such \alpha < \lambda\ra$
has at most two entries and their intersection has size $\kappa$.
\end{proof}

In the special case $F_\beta = \{\nu_\beta \concat \la j \ra \such j \in c_\beta\}$, the construction
of Lemma~\ref{fusionstark2} coincides with the simpler construction from Lemma~\ref{fusionstark}.

\begin{definition}\label{QcT}
We assume $\bQ^1_\kappa$ collapses $2^\kappa$ to $\omega$.  Let $\name{\tau}$ and ${\mathcal T} = \la (a_\eta,n_\eta,\varrho) \such \eta \in {}^{\omega>}(2^\kappa)\ra$
  be as in Lemma \ref{inserted}.
Now let $Q_{\cT}$ be the set of $\kappa$-Miller trees $p$ such that for every 
$\nu \in \splitt(p)$ there is $\eta_{p,\nu} = \eta_\nu \in {}^{\omega>} (2^\kappa)$
such that 
\begin{equation}\label{witness}
  \set_p(\nu)
= \{ \eps \in \kappa \such \nu \concat \la \eps \ra \in p \} = a_{\eta_\nu}.
\end{equation}
\nothing{\item[(b)] $|\{\varrho_{\eta_\nu \concat \la \alpha \ra} \rest n_\eta \such 
  \alpha < 2^\kappa \wedge \set_p(\nu) \cap a_{\eta_\nu \concat \la \alpha \ra} \in [\kappa]^\kappa \}| \geq 2$.
\footnote{ We could get $2^\kappa$ instead of $2$ here.}}
\end{definition}

By the properties of $\cT$, the node $\eta_{p,\nu}$ is unique.

\begin{lemma}\label{QcTdense}
Assume that $\bQ^1_\kappa$ collapses $2^\kappa$ to $\omega$, let
 $\cT$  be chosen as in Lemma~\ref{inserted}, and let $Q_{\cT}$
be defined from $\cT$ as above. Then $Q_{\cT}$ is dense in $\bQ^2_\kappa$.
\end{lemma}

\begin{proof}
Let $p_0 = T \in \bQ^2_\kappa$.  Let $\la \nu_\alpha \such \alpha<\kappa^{<\kappa} \ra$ be an injective enumeration of $\kappa^{<\kappa}$ with property \eqref{gut}.
We now define fusion sequence
$\la p_\alpha,\nu_\alpha, c_\alpha \such \alpha \leq \kappa^\kappa\ra$ according
to the pattern in Lemma \ref{fusionstark} in order to find
$p_{\kappa^{<\kappa}} \geq T$ such that $p_{\kappa^{<\kappa}} \in Q_{\cT}$.

Suppose that $p_\alpha$ and $\nu_\alpha$ are given.
If $\nu_\alpha $ is not in $p_\alpha$ or is not a
splitting node in $p_\alpha$,
then we let $p_{\alpha+1} = p_\alpha$.
If $\nu_\alpha \in \splitt(p_\alpha)$, then
according to Lemma~\ref{inserted} clause (d) there is
$\eta \in {}^{\omega>} (2^\kappa)$ such that $\suc_{p_\alpha}(\nu_\alpha) \supseteq a_\eta$.
We choose such an  $\eta$ of minimal length and call it $\eta(\alpha)$.

Then we strengthen $p_\alpha$ to
\begin{equation}\label{allowed}
  \begin{split}
    p_{\alpha+1} = & \bigcup\{p_\alpha^{\la \nu' \ra } \such \nu' = \nu_\alpha \concat \la i \ra \wedge i \in a_{\eta(\alpha)}\}
  \cup \\
  & \bigcup\{p_\alpha^{\la \eta \ra} \such \eta \not\trianglelefteq \nu_\alpha
      \wedge \nu_\alpha \not\trianglelefteq \eta\}
.\end{split}
  \end{equation}

Now we have that
\[\eta_{p_{\alpha+1},\nu_\alpha}=\eta(\alpha), c_\alpha = a_{\eta(\alpha)}.\]
\nothing{
We cannot change it sidewards, by almost disjointness of the $a_\eta$.
We cannot replace $\eta$ by a proper extension, since $\suc_{p_{\alpha+1}}(\nu_\alpha)= a_\eta$ is not a subset of any of the $a_{\eta \concat \la \gamma \ra}$. We cannot replace $\eta$
by one of its proper initial segments, since
such a segment $\eta \rest i$, $i < \lg(\eta)$ fulfils
$a_{\eta \rest i} \not\subseteq \set_{p_{\alpha}}(\nu_\alpha)$
and a fortiori  $a_{\eta \rest i} \not\subseteq \set_{p_{\alpha+1}}(\nu_\alpha)$, since we
chose the shortest $\eta$.
}
For limit ordinals $\lambda \leq \kappa^{<\kappa}$, we let $p_\lambda = \bigcap\{p_\beta \such \beta < \lambda\}$. Since the sequence
$\la p_\alpha, \nu_\alpha, c_\alpha \such \alpha \leq \kappa^{<\kappa} \ra$ matches the pattern in Lemma \ref{fusionstark}, we have $p_{\kappa^{<\kappa}} \in \bQ^2_\kappa$.
By construction, for any
$\alpha < \kappa^{<\kappa}$ for any $\delta \in[\alpha+1,\kappa^{<\kappa})$, $\nu_\alpha \in \splitt(p_\delta)$ implies
  \[
  \set_{p_{\alpha+1}}(\nu_\alpha) = \set_{p_\delta}(\nu_\alpha)=a_{\eta(\alpha)}.
  \]
  Hence the condition $p=p_{\kappa^{<\kappa}}$ fulfils Equation~\eqref{witness}
  in its splitting node $\nu_\alpha$ with  witness
$\eta_{p,\nu_\alpha}= \eta(\alpha)$. Since all nodes are
enumerated, we have $p_{\kappa^{<\kappa}} \in Q_{\cT}$.
 \end{proof}

We use only the inclusion $\set_p(\nu) \subseteq a_{\eta_\nu}$ from Definition \ref{QcT}.

\begin{definition}\label{varrho} We assume that $\bQ^1_\kappa$ collapses $2^\kappa$ to $\omega$ and the ${\mathcal T}$ is as in Lemma~\ref{inserted}.
For $T \in Q_{\cT}$ and a  splitting node $\nu$
of $T$ we set $\varrho_{T,\nu} := \varrho_{\eta_{T,\nu}} \in {}^{\omega>} (2^\kappa)$.
Recall $\eta_{T,\nu}$ is defined in Def.~\ref{QcT}, and $\varrho$ is
a component of ${\mathcal T}$.
\end{definition}

For $p \in Q_{\cT}$, the relation $\nu \trianglelefteq \nu' \in p$ 
does neither imply $\eta_{\nu} \trianglelefteq \eta_{\nu'}$ nor $\varrho_\nu \trianglelefteq \varrho_{\nu'}$.
However, $\eta_\nu \triangleleft \eta_{\nu'}$ implies $a_{\eta_\nu} \supset a_{\eta_{\nu'}}$ and $ \varrho_\nu \triangleleft \varrho_{\nu'}$.

\begin{observation} We assume that $\bQ^1_\kappa$ collapses $2^\kappa$ to $\omega$.
  Let $p_1, p_2 \in Q_{\cT}$.
If $p_1 \leq_{\bQ^2_\kappa} p_2$ then 
for $\nu \in \splitt(p_2)$ we have $\nu \in \splitt(p_1)$ and $\varrho_{p_1,\nu} \trianglelefteq \varrho_{p_2,\nu}$.
\end{observation}

We introduce dense sets:

\begin{definition}\label{Dn}  We assume that $\bQ^1_\kappa$ collapses $2^\kappa$ to $\omega$.
    Let $n \in \omega$.
    \begin{equation*}
          D_n = \bigl\{p \in Q_{\cT} \such  (\forall \nu \in \splitt(p)) 
      (\lg(\varrho_{p,\nu}) > n)\}.
      \end{equation*}
\end{definition}

$D_n$ is open dense in $Q_{\cT}$ and the intersection of the $D_n$ is empty.
The following technical lemma is the first step of a transformation of a $\bQ^1_\kappa$-name of a surjection from $\omega$  onto $2^\kappa$ into a
$\bQ^2_\kappa$-name of such a surjection.

\begin{lemma}\label{long2}
  We assume that $\bQ^1_\kappa$ collapses $2^\kappa$ to $\omega$, $\cf(\kappa) > \omega$ and $2^{(\kappa^{<\kappa})} = 2^\kappa$.
 Let $\la T_{\alpha} \such \alpha < 2^\kappa \ra$ enumerate $\bQ^2_\kappa$ such that each Miller tree appears $2^\kappa$ times.
There is $\la (p_\alpha, n_\alpha, \bar{\gamma}_\alpha) \such \alpha < 2^\kappa\ra$ such that
\begin{myrules}
\item[(a)] $n_\alpha < \omega$,
\item[(b)] $p_\alpha \in D_{n_\alpha}$ and $p_\alpha \geq T_\alpha$. 
\item[(c)] If $\beta < \alpha$ and $n_\beta \geq n_\alpha$ then $p_\beta \perp p_\alpha$.
\item[(d)]
  $\bar{\gamma}_\alpha = \la \gamma_{\alpha,\nu} \such \nu \in \splitt(p_\alpha)\ra$.
\item[(e)] $(\forall \nu \in \splitt(p_\alpha))
  (a_{\eta_{p_\alpha,\nu}} \Vdash_{\bQ^1_\kappa} \gamma_{\alpha,\nu} \in \rge(\varrho_{p_\alpha,\nu}))$.

 \item[(f)]
  $\gamma_{\alpha,\nu} \in 2^\kappa \setminus W_{<\alpha,\nu}$ with
  \[W_{<\alpha,\nu} = \bigcup\{\rge(\varrho_{p_\beta,\nu}) \such \beta < \alpha,
  \nu \in \splitt(p_\beta)\}.\]
\end{myrules}
\end{lemma}

  \begin{proof}
  Assume that $\la (p_\beta, n_\beta, \bar{\gamma}_\beta) \such \beta < \alpha \ra$ has been defined and we are to define $(p_\alpha,n_\alpha, \bar{\gamma}_\alpha)$.
  Note that the $p_\beta$ need not be increasing in strength.
  
  \begin{myrules}
  \item[$(\oplus)_1$]
    The choice of the $a_\eta$ in Lemma~\ref{inserted}
    and the choice $Q_{\cT}$ and of $\eta_{p_\beta,\nu}$ for $\nu \in \splitt(p_\beta)$, $\beta < \alpha$,
    imply that the set $W_{<\alpha,\nu}$ is
    well defined and of cardinality $\leq |\alpha| + \aleph_0 < 2^\kappa$.
    Hence we can choose $\gamma_{\alpha,\nu} \in 2^\kappa \setminus W_{<\alpha,\nu}$.

  \item[$(\oplus)_2$]
    With the fusion Lemma~\ref{fusionstark}
    we choose $q_\alpha \geq T_\alpha$,
    $q_\alpha \in Q_{\cT}$, such that
    \[
    (\forall \nu \in \splitt(q_\alpha))(a_{\eta_{q_\alpha,\nu}} \Vdash_{\bQ^1_\kappa}
    \gamma_{\alpha,\nu} \in \rge(\varrho_{q_\alpha,\nu})).
      \]
\nothing{
  \marginparr{read}
  {\sf We do not need that \[\bQ^1_\kappa \Vdash \forall \gamma \in 2^\kappa \exists^\infty n \name{\tau}(n) = \gamma,\]
  since we do not require that $a_{\eta_{p_\alpha,\nu}}$ forces $\name{\tau}(n) = \gamma$
  for a particular $n$. }
\nothing{The set $\{n_\beta \such \beta < \alpha\}$ could already be infinite. There is no particular candidate for fixing an $n$, since the length of
  $\varrho_{p_\alpha,\nu}$ depends on $\nu$.}
 }
\item[$(\oplus)_3$]Let $q \in \bQ^2_\kappa$. For $n \in \omega$ and $\nu \in \splitt(q)$ we let
          \[\cU_{\alpha,\nu,n}(q) = \{ \beta < \alpha \such n_\beta = n , \nu \in \splitt(p_\beta) \wedge
          |\set_{q}(\nu) \cap \set_{p_\beta}(\nu) | = \kappa \}.
          \]
          \[\cU_{\alpha, \nu}(q) = \bigcup \{\cU_{\alpha,\nu,n}(q) \such n \in \omega \}.\]
\item[$(\oplus)_4$]
  \begin{myrules}
    \item[(a)] If $n \in \omega$ and $ \nu \in \splitt(q_\alpha)$ then
  \[
  \beta \in \cU_{\alpha,\nu}(q_\alpha) \rightarrow \varrho_{p_\beta,\nu} \trianglelefteq \varrho_{q_\alpha,\nu}.
  \]
  This is seen as follows. We let $a= \set_{p_\beta}(\nu) \cap \set_{q_\alpha}(\nu)$. Since $\beta \in \cU_{\alpha,\nu}(q_\alpha)$, $a \in [\kappa]^\kappa$.
  Clearly $a \Vdash_{\bQ^1_\kappa} \name{\tau} \triangleright \varrho_{p_\beta,\nu}, \varrho_{q_\alpha,\nu}$. So either  $\varrho_{p_\beta,\nu} \triangleleft
  \varrho_{q_\alpha,\nu}$
  or  $\varrho_{p_\beta,\nu} \trianglerighteq  \varrho_{q_\alpha,\nu}$.
  However, since $\gamma_{\alpha,\nu} \in \rge(\varrho_{q_\alpha,\nu}) \setminus W_{<\alpha,\nu}$, only $\varrho_{q_\alpha,\nu} \triangleright \varrho_{p_\beta,\nu}$ is possible.
\item[(b)]  
  So for $\nu \in \splitt(q_\alpha)$, the set $\{ \varrho_{p_\beta,\nu} \such \beta \in \cU_{\alpha,\nu}(q_\alpha)\}$ has at most
  $\lg(\varrho_{q_\alpha,\nu})$ elements.

\item[(c)] 
  The assigment $\beta \mapsto \varrho_{p_\beta,\nu}$  is
  is defined between $\cU_{\alpha,\nu}(q_\alpha)$ and  $\{ \varrho_{p_\beta,\nu} \such \beta \in \cU_{\alpha,\nu}(q_\alpha)\}$. According to
  properties (e) and (f) in the induction hypothesis, the assigment is
  injective, and hence \\
$|\cU_{\alpha,\nu}(q_\alpha) |\leq \lg(\varrho_{q_\alpha,\nu})$.

\item[(d)] We state for further use that $\cU_{\alpha,\nu}(q_\alpha)$ is finite
  and for any $q \geq q_\alpha$, $\cU_{\alpha,\nu}(q) \subseteq \cU_{\alpha,\nu}(q_\alpha)$.

  \end{myrules}
\item[$(\oplus)_5$]
  We look at the cone above $q_\alpha$ and show:
  \begin{equation}\label{oplus5}
    \begin{split}
      &(\forall q \geq q_\alpha)(\forall \nu \in \splitt(q)) (\exists r_{\alpha,\nu}\geq_{\bQ^2_\kappa} q)\\
      &(\exists c \in [\set_q(\nu)]^\kappa)(\exists F \subseteq \{\eta \in \splitt(q) \such \eta \triangleright \nu\})\\
      & \bigl(r_{\alpha,\nu} = q(\nu,c,F)   
     \wedge (\forall \beta \in \cU_{\alpha,\nu}(q_\alpha))
     (r^{\la \nu \ra}_{\alpha,\nu} \perp p_\beta^{\la \nu \ra} \vee p_\beta^{\la\nu\ra} \leq r^{\la \nu \ra}_{\alpha,\nu})\bigr).
    \end{split}
  \end{equation}
 
  How do we find $r_{\alpha,\nu}= r_{\alpha,\nu}(q)$? Given $q \geq_{\bQ^2_\kappa} q_\alpha$, $\nu \in \splitt(q)$ we enumerate $\cU_{\alpha,\nu}(q_\alpha)$  as
  $\beta_0$, \dots, $\beta_{k-1}$.
  We let $r_{0} = q$ and by induction on $i \leq k$ we define $r_i$, increasing in strength,  with $\nu \in \splitt(r_i)$ and $c_i = \set_{r_i}(\nu)$.
  Thus the $c_i$ are $\subseteq$-decreasing sets of size $\kappa$.
 Given $r_i$, we distinguish cases:

 First case: $\beta_i \not\in \cU_{\alpha,\nu}(r_i)$. Then there is
 $c_{i+1} \in [\set_{r_i}(\nu)]^\kappa $, $c_{i+1} \cap \set_{p_{\beta_i}}(\nu) = \emptyset$. We let $r_{i+1} = r_i(\nu, c_{i+1})$ and thus have $r_{i+1}^{\la \nu \ra} \perp p_{\beta_i}$.

 Second case: $\beta_i \in \cU_{\alpha,\nu}(r_i)$.
 We let
 \[c_i = \{j \in \set_{r_i}(\nu) \such r_i^{\la \nu \concat \la j \ra \ra} \geq p_{\beta_i}^{\la \nu \concat \la j \ra \ra}\} \cup \{j \in \set_{r_i}(\nu) \such r_i^{\la \nu \concat \la j \ra \ra} \not\geq p_{\beta_i}^{\la \nu \concat \la j \ra \ra}\}.\]

If $c_{i,1} =   \{j \in \set_{r_i}(\nu) \such r_i^{\la \nu \concat \la j \ra \ra} \geq p_{\beta_i}^{\la \nu \concat \la j \ra \ra}\} $ has size $\kappa$, then we let $c_{i+1} = c_{1,i}$ and $r_{i+1} = r_i(\nu,c_{i+1})$ and thus get $r_{i+1}^{\la \nu \ra} \geq p_{\beta_i}$.

If $|c_{i,1}|< \kappa$, then $c_{i,2} =   \{j \in \set_{r_i}(\nu) \such r_i^{\la \nu \concat \la j \ra \ra} \not\geq p_{\beta_i}^{\la \nu \concat \la j \ra \ra}\} $ has size $\kappa$,
\nothing{
either $r_i^{\la \nu \ra} \geq p_{\beta_i}^{\la \nu \ra}$, in which case we let $c_{i+1} = c_i$,  or there is
  $c_{i+1} \in [c_i]^\kappa$ and there are fronts $F_{\nu,i,j}$ in $r_{i+1}$ above $\nu \concat \la j \ra $ for $j \in c_{i+1}$ such that 
  $r_{i+1} = q(\nu,c_{i+1}, F_{\nu, i})$ fulfils $r_{i+1}^{\la \nu \ra} \perp p_{\beta_i}^{\la \nu \ra}$. In the end we take $r_{\alpha,\nu} = r_k$.
}
and we let $c_{i+1} = c_{i,2}$.
For $j \in c_{i+1}$,  $r_i^{\la \nu\concat \la j \ra\ra} \not\geq p_{\beta_i}^{\la \nu \concat \la j\ra \ra}$. Thus we
  can find a node in the $r_i^{\la \nu\concat \la j \ra\ra}\setminus  p_{\beta_i}^{\la\nu \concat \la j\ra\ra}$ and above this node we find a splitting node of $r_i$. We take this latter splitting node
  into $r_{i+1}$ as the direct successor splitting node to $\nu \concat \la j \ra$. Doing so for every $j \in c_{i+1}$ we get $F_{\nu,i}$, a front
  strictly above $\nu$ in $r_{i+1}= r_i(\nu,c_{i+1}, F_{\nu,i})$.
  Again we get $r_{i+1}^{\la \nu \ra} \perp p_{\beta_i}$.

  In the end we let $r_{\alpha,\nu} = r_k$. There is a front $F$
  that contains for each $j \in c_k$  the shortest splitting node of $r_k$ above $\nu \concat \la j \ra$. Thus we have $r_k = r_{\alpha,\nu}= q(\nu,c_k, F)$
  and $r_{\alpha,\nu}$ fulfils \eqref{oplus5}.

\item[$(\oplus)_6$] Now we use $(\oplus)_5$ iteratively along all $\nu \in \kappa^{<\kappa} $ to find
  a fusion sequence $\la r_{\alpha,\nu},\nu,c_\nu, F_{\nu} \such \nu < \kappa^{<\kappa}\ra$  with starting point $q_\alpha = r_{0,\nu_0}$. In this sequence, $r_{\alpha, \nu} $ is chosen as $r_{\alpha,\nu}(q)$ in $\oplus_5$ for $q = \bigcap_{\beta < \alpha} r_\beta$, if $\nu \in \splitt(q)$. If $\nu \not\in \splitt(q)$, then  $r_{\alpha,\nu} = q$.
  Then we apply the fusion Lemma \ref{fusionstark2} and get an upper bound $r_\alpha$ of $r_{\alpha,\nu}$, $\nu \in {}^{\kappa>}\kappa$.
Note $r_\alpha^{\la \nu \ra} \perp p_\beta$ iff  $r_\alpha^{\la \nu \ra} \perp p_\beta^{\la \nu \ra}$ and 
$r_\alpha^{\la \nu \ra} \geq p_\beta$ iff  $r_\alpha^{\la \nu \ra} \geq p_\beta^{\la \nu \ra}$.
  Hence $r_\alpha \geq q_\alpha$ and 
  \[
  (\forall \nu \in \splitt(r_\alpha))( \forall \beta \in \cU_{\alpha,\nu}(q_\alpha))
  (r_\alpha^{\la \nu \ra} \perp p_\beta \vee p_\beta \leq r_\alpha^{\la \nu \ra}).
  \]
\item[$(\oplus)_7$] Finally we choose $n_\alpha$ and $p_\alpha$.
There are $k$ and $\nu$ such that $n < \omega$ and $ \nu \in \splitt(r_\alpha)$ 
such that $p_\alpha = r_\alpha^{\la \nu \ra}$ fulfils
\[(\forall \beta < \alpha)( n_\beta \geq  k \rightarrow p_\alpha \perp p_\beta).
\]
Proof of existence. By induction on $k \in \omega$ we try to find
 $\la \nu_k, \beta_k \such k \in \omega \ra$ such that
\begin{myrules}
\item[(a)] $\nu_k \in \splitt(r_\alpha)$,
\item[(b)] $\nu_k \triangleleft \nu_m$ for $k <m$,
\item[(c)] $\beta_k < \alpha$ and $n_{\beta_k} \geq k$ and $r_\alpha^{\la \nu_k \ra} \geq p_{\beta_k}$.
\end{myrules}
If we succeed, then $\nu_\ast=\bigcup \{\nu_k \such k \in \omega\} = \nu^*
\in \splitt(r_\alpha)$ by Definition~\ref{1.1} (2).
Here we use that $\cf(\kappa) > \omega$. Hence
\begin{equation*}
  \begin{split}
    r_\alpha^{\la \nu^* \ra} & \in Q_{\cT} \cap \bigcap \{D_k \such k<\omega\} \mbox{ and}\\
a_{\eta_{r_\alpha, \nu^*}} & \mbox{ determines in }\Vdash_{\bQ^1_\kappa}
\mbox{ for any }k < \omega \mbox{ the value of }\name{\tau}\rest k.
\end{split}\end{equation*}
This is a contradiction.

So there is a smallest $k$ such that $\nu_k$ cannot be defined.
We let $n_\alpha= k$.
We let $p_\alpha$ be a strengthening of $r_\alpha^{\la \nu_{k-1} \ra}$ such
that $p_\alpha \in D_{n_\alpha}$. For finding such a strengthening
we again invoke the fusion Lemma~\ref{fusionstark}. 

We show that $p_\alpha \perp p_\beta$ for $\beta < \alpha$ with $n_\beta \geq k$.
Otherwise, having arrived at $r_\alpha^{\la \nu_{k-1}\ra}$ we find some $\beta_k , \alpha$ such that $n_{\beta_k} \geq k$ and $r_\alpha^{\la \nu_{k-1}\ra} $ is compatible
  with $p_{\beta_k}$. Then we can prolong $\nu_{k-1}$ to a splitting node
  $\nu_k \in \splitt(p_{\beta_k}) \cap \splitt(r_{\alpha})$. By the choice of $r_\alpha$ the latter implies that $r_{\alpha}^{\la \nu_k \ra} \geq p_{\beta_k}$. However, now we would have found $\nu_k, \beta_k$ as required in contradiction to the choice of $k$.
  \end{myrules}
  \end{proof}  
\begin{remark}
  Conditions (a) to (c) of  Lemma \ref{long2} yield:
  For any $k <\omega$,
  \[\{p_\alpha \such n_\alpha \geq k\} \mbox{ is dense in } \bQ^2_\kappa.\]
\end{remark}
\begin{proof} Let $k$ and $p $ be given. There is
  $\alpha_0 $ such that $T_{\alpha_0} \in D_0$ and $T_{\alpha_0} \geq_{\bQ^2_\kappa} p$.
  Then $p_{\alpha_0} \geq T_{\alpha_0}$ and $n_{\alpha_0}$.
 Then there is
 $\alpha_1> \alpha_0$ such that $T_{\alpha_1} \geq_{\bQ^2_\kappa} p_{\alpha_0}$.
 Then $p_{\alpha_1} \geq T_{\alpha_1}$ and hence
 by condition (c), $n_{\alpha_1} >n_{\alpha_0}\geq 0$.
 We can
  can repeat the argument $k-1$ times.
\end{proof}

  Now we drop the component $\bar{\gamma}_\alpha$  from a sequence
  $\la p_\alpha, n_\alpha, \bar{\gamma}_\alpha \such \alpha < 2^\kappa \ra$
  given by Lemma~\ref{long2}. Then we get a sequence with properties (a), (b), and a weakening (c)  with the property stated in the remark.

\begin{lemma}\label{short}  We assume that $\bQ^1_\kappa$ collapses $2^\kappa$ to $\omega$, $\cf(\kappa) > \omega$ and $2^{(2^{<\kappa})} = 2^\kappa$.
Let 
  $\la T_{\alpha}  \such \alpha < 2^\kappa\ra$ enumerate all Miller trees that such each tree appears $2^\kappa$ times.
  If $\la (p_\alpha, n_\alpha) \such \alpha < 2^\kappa\ra$  are
  such that
\begin{myrules}
\item[(a)] $n_\alpha < \omega$,
\item[(b)]
  $p_\alpha \in D_{n_\alpha}$ and $p_\alpha \geq T_\alpha$, 
\item[(c)] if $\beta < \alpha$ and $n_\beta = n_\alpha$ then $p_\beta \perp p_\alpha$,
\item[(d)]   for any $k \in \omega$,
  $\{p_\alpha \such n_\alpha \geq k\} $ is dense in $\bQ^2_\kappa$.
  \end{myrules}
Then there is a $\bQ^2_\kappa$-name
  $\name{\tau}'$ for a surjection of $\omega$ onto $2^\kappa$.
  \end{lemma}

\proof
Let $G$ be a $\bQ^2_\kappa$-generic filter over ${\bV}$.
We define $\name{\tau}(n)$, a $\bQ^2_\kappa$-name by $\name{\tau}(n)[G] = 
\alpha$ if $p_\alpha \in G$ and $n_\alpha = n$.
The name $\name{\tau}$ is a name of a function  by (c).
By (d), the domain of $\name{\tau}$ is forced to be infinite.
For any $p \in \bQ^2_\kappa$ we let $U_p = \{ \alpha \such T_\alpha= p\}$.
$U_p$ is of size $2^\kappa$, in particular for
$\alpha \in 2^\kappa$ we have  $|U_{p_\alpha}|= 2^\kappa$.
Hence there is $f \colon 2^\kappa \to 2^\kappa$ such that
for any $\alpha, \gamma \in 2^\kappa$ and $\exists
\beta \in U_{p_\gamma}$ with $f(\beta)
= \alpha$. We let $\name{\tau'}(n) = f(\name{\tau}(n))$.
Next we show
\[\bQ^2_\kappa \Vdash \rge(\name{\tau'}) = 2^\kappa.
\]
Suppose $p \in Q_{\cT}$  and $\alpha < 2^\kappa$ are given.
By construction the sequence
$\{p_\beta \such \beta < 2^\kappa\}$ is dense.
Let $p \leq p_\gamma$. Then there is  $\beta \in U_{p_\gamma}$, with $f(\beta) = \alpha$. However, 
$\beta \in U_{p_\gamma}$ means $T_\beta = p_\gamma  \leq p_\beta$
by construction. By the definition of $\name{\tau}$,
$p_\beta \Vdash \name{\tau}(n_\beta) = \beta$, so
$p_\beta \Vdash f(\name{\tau}(n_\beta)) = \alpha$.
\proofend

So we can sum up:

\begin{theorem}\label{main}
  We assume that $\bQ^1_\kappa$ collapses $2^\kappa$ to $\omega$ and $\cf(\kappa) > \omega$ and $2^{(\kappa^{<\kappa})} = 2^\kappa$.
  Then the forcing with
 $\bQ_\kappa^2$ collapses $2^\kappa$ to $\aleph_0$.
\end{theorem}

  \section{$\kappa$-Cohen reals and the Levy collapse}
\label{S4}

Another vice of a $\kappa$-tree forcing is to add $\kappa$-Cohen reals.
In this section we show that under the above conditions,
$\bQ^\kappa_2$ adds Cohen reals and is equivalent to the Levy collapse
of $2^\kappa$ to $\aleph_0$.

\begin{lemma} If $\bP$ collapses $2^\kappa$ to $\aleph_0$, $\cf(\kappa) > \aleph_0$, and $2^{2^{<\kappa}} = 2^\kappa$, then $\bQ_\kappa^2$ adds a
$\kappa$-Cohen real.
\end{lemma}
\proof
Let $G $ be $\bQ^2_\kappa$-generic over $\bV$. Let $f \colon \omega \to 2^{<\kappa}$ be a function in ${\bf V}[G]$, such that 
$(\forall \eta \in 2^{<\kappa})(\exists^\infty k f(k) = \eta)$.
Such a function exists since $2^{< \kappa} \leq 2^\kappa$.  

Since $2^{2^{<\kappa}} = 2^\kappa$, we can enumerate all antichains in $\bC(\kappa)$ in $\alpha_\ast \leq 2^\kappa$ many steps.
In $\bV[G]$, $\alpha_*$ is countable. We list it as $\la \alpha_n \such n < \omega \ra$. Now we choose $\eta_n \in \bC(\kappa)^{\bV}$ by induction on $n$ in $\bV[G]$: 
$\eta_0 = \emptyset$. Given $\eta_n$ we choose 
$k_n$ such that $f(k_n) = \eta_n$ and then we choose $\eta_{n+1} \trianglerighteq \eta_n$, such that $\eta_{n+1} \in I_{\alpha_n}$. Then $\{\eta \such (\exists  n < \omega) (\eta \trianglelefteq f(k_n))\}$ is a $\bC(\kappa)$-generic filter over $\bV$ and it exists in $V[G]$, since it it definable from $\{f(k_n) \such n < \omega\}$.
\proofend

Two forcings $\bP_1$, $\bP_2$ are said to be equivalent if
their regular open algebras $\RO(\bP_i)$ coincide
(for a definition of the regular open algebra of a poset, see, e.g., \cite[Corollary 14.12]{Jech3}). Some forcings are characterised up to equivalence
just by their size and their collapsing behaviour. 

\begin{definition}
  Let $B$ be a Boolean algebra. We write $B^+ = B \setminus \{0\}$.
  A subset $D \subseteq B^+$ is called \emph{dense} if $(\forall b \in B^+)(\exists d \in D)(d \leq b)$.
The \emph{density} of a Boolean algebra $B$ is the least size of a dense subset
of $B$. A Boolean algebra $B$ has uniform density if for every $a \in B^ +$, $B \rest a$ has the same density.
The \emph{density} of a forcing order $(\bP,<)$ is the density of the
regular open algebra $\RO(\bP)$.
\end{definition}

\begin{lemma}\label{levy1} \cite[Lemma 26.7]{Jech3}. Let $(Q, <)$ be a notion of forcing such that $|Q| = \lambda > \aleph_0$ and
such that $Q$ collapses $\lambda$ onto $\aleph_0$ , i.e.,
\[0_{Q} \Vdash_Q |\check{\lambda}| = \aleph_0.\]
Then ${\rm RO}(Q) = \Levy(\aleph_0,\lambda)$.
\end{lemma}

\nothing{
\begin{lemma}\cite[Corollary 26.8 (Kripke)]{Jech3}. If $B$ is a complete Boolean algebra and $|B| \leq\lambda$
then $B$ embeds as a complete subalgebra into the algebra $\Levy(\aleph_0,\lambda)$.
\end{lemma}
}
\begin{lemma} If
  $\bQ_\kappa^1$ collapses $2^\kappa$ to $\aleph_0$,
then $\bQ_\kappa^1$ is equivalent of $\Levy(\aleph_0, 2^\kappa)$.
\end{lemma}
\proof $\bQ^1_\kappa$ has size $2^\kappa$. Hence Lemma \ref{levy1} yields ${\rm RO}(\bQ^1_\kappa)= {\rm Levy}(\aleph_0,2^\kappa)$. 
\proofend

\begin{definition}
  A Boolean algebra is \emph{$(\theta,\lambda)$-nowhere distributive} if there are antichains $\bar{p}^\eps = \la p^\eps_\alpha \such \alpha < \alpha_\eps\ra$ of $\bP$ for $\eps < \theta$ such that for every $p \in \bP$ for some $\eps < \theta$
  \[|\{\alpha < \alpha_\eps \such p \not\perp p^\eps_\alpha\}| \geq \lambda.\]
\end{definition}

\begin{lemma}\label{balcar} \cite[Theorem 1.15]{BalcarSimon89}
Let  $\theta < \lambda$ be regular cardinals.
  \begin{myrules}
  \item[(1)] Suppose that  $\bP$ has the following properties (a) to (c).
    \begin{myrules}
    \item[(a)] $\bP$ is a $(\theta,\lambda)$-nowhere distributive forcing notion,
    \item[(b)] $\bP$ has density $\lambda$,
    \item[(c)] in case $\theta > \aleph_0$, $\bP$ has a $\theta$-complete subset
      $S$. The latter means: $(\forall B \in [S]^{<\theta})(\exists s\in S)(\forall b \in B)(b \leq_{\bP} s)$.
    \end{myrules}
    \emph{Then} $\bP$ is equivalent to $\Levy(\theta,\lambda)$.
  \item[(2)] Under (a) and (b) $\bP$ collapses $\lambda$ to $\theta$
    (and may or may not collapse $\aleph_0$).
  \end{myrules}
  \end{lemma}

\begin{proposition}
If there is a $\kappa$-mad family of size $2^\kappa$ the forcing
$\bQ^1_\kappa$ is $(\aleph_0,2^\kappa)$-nowhere distributive.
\end{proposition}
\begin{proof}
  Lemma \ref{inserted}  gives $\cT$ such that  
  $\bar{p}^n = \{ a_\eta \such \eta \in {}^n (2^\kappa)\}$, $n \in \omega$,
  witnesses $(\aleph_0,2^\kappa)$-nowhere distributivity.
\end{proof}

By Lemma~\ref{levy1} and Theorem~\ref{main} we get:

\begin{proposition}
  If $\bQ^1_\kappa$ collapses $2^\kappa$ to $\aleph_0$, $\cf(\kappa) > \aleph$ and
  and $2^{(\kappa^{<\kappa})} = 2^\kappa$ then $\bQ^2_\kappa$ is equivalent to ${\rm Levy}(\aleph_0,2^\kappa)$.
\end{proposition}

\nothing{
\begin{lemma}
   $\Levy(\aleph_0, 2^{<\kappa})$
adds a $\kappa$-Cohen real. 
\end{lemma}

Indeed, something more holds
\begin{lemma}\label{d5} 
If $\theta= \cf(\theta) < \kappa$ then the forcing with $\Levy(\theta, 2^{<\kappa})$ adds a $\kappa$-Cohen real.
  \end{lemma}

\proof 
Let $\theta= \cf(\theta) < \kappa$. Let $\{\varrho_\alpha \such \alpha < 2^{<\kappa}\}$ list $2^{<\kappa}$.
For $\nu \in {}^{<\theta} (2^{<\kappa})$ we let $\varrho(\nu)$ be the concatenation of the $\varrho_{\nu(\delta)}$,  $\delta \in \dom(\nu)$.
Let $G$ be a ${\rm Levy}(\theta, 2^{<\kappa})$-generic filter.
We define a filter $H(G)  \subseteq \bC_\kappa$ as follows:
\[p \in H(G) \leftrightarrow (\exists \nu \in G)(p \trianglelefteq \varrho(\nu)).\]
It is clear that $H(G)$ is a filter. We show that it is $\bC_\kappa$ generic.
Let $D \subseteq \bC_\kappa$ be in the ground model and be open and  dense.
Let $D' = \{ \nu \in {\rm Levy}(\theta, 2^{<\kappa}) \such \varrho(\nu) \in D\}$.
$D'$ is dense in ${\rm Levy}(\theta, 2^{<\kappa})$. Hence there is a
$\nu \in G \cap D'$.
\proofend
}

{\bf Acknowledgement}: We thank Marlene Koelbing for pointing out a gap in an earlier version.
\def\cprime{$'$} \def\germ{\frak} \def\scr{\cal} \ifx\documentclass\undefinedcs
  \def\bf{\fam\bffam\tenbf}\def\rm{\fam0\tenrm}\fi 
  \def\defaultdefine#1#2{\expandafter\ifx\csname#1\endcsname\relax
  \expandafter\def\csname#1\endcsname{#2}\fi} \defaultdefine{Bbb}{\bf}
  \defaultdefine{frak}{\bf} \defaultdefine{=}{\B} 
  \defaultdefine{mathfrak}{\frak} \defaultdefine{mathbb}{\bf}
  \defaultdefine{mathcal}{\cal}
  \defaultdefine{beth}{BETH}\defaultdefine{cal}{\bf} \def\bbfI{{\Bbb I}}
  \def\mbox{\hbox} \def\text{\hbox} \def\om{\omega} \def\Cal#1{{\bf #1}}
  \def\pcf{pcf} \defaultdefine{cf}{cf} \defaultdefine{reals}{{\Bbb R}}
  \defaultdefine{real}{{\Bbb R}} \def\restriction{{|}} \def\club{CLUB}
  \def\w{\omega} \def\exist{\exists} \def\se{{\germ se}} \def\bb{{\bf b}}
  \def\equivalence{\equiv} \let\lt< \let\gt>


\end{document}